\documentclass[12pt,reqno]{article}
\usepackage{amscd,  amsmath, amssymb, amsthm, sectsty}

\newtheorem{thm}{Theorem}
\newtheorem{lem}{Lemma}
\sectionfont{\normalsize}

\theoremstyle{definition}
\newtheorem*{rem}{Remark}

\title{Keeler's Theorem and Products of Distinct Transpositions}

\author{\\ Ron Evans, Lihua Huang, and Tuan Nguyen}
\date{}
\begin{document}
\maketitle
%%% Abstract %%%%%
\begin{abstract}
An episode of 
{\em Futurama} features a two-body mind-switching machine which
will not work more than
once on the same pair of bodies.
After the Futurama community engages in a mind-switching spree,
the question is asked,
``Can the switching be undone so as to restore all minds to their
original bodies?"  
Ken Keeler found an algorithm that undoes any mind-scrambling
permutation with the aid of two ``outsiders."  We refine Keeler's result by
providing a more efficient algorithm that uses the smallest possible number of
switches.  We also present best possible algorithms for undoing two
natural sequences of switches, each 
sequence effecting a cyclic mind-scrambling 
permutation in the symmetric group $S_n$.  
Finally, we give necessary and sufficient conditions 
on $m$ and $n$ for the identity permutation to be expressible as
a product of $m$ distinct transpositions in $S_n$.
\end{abstract}

\section{Introduction}

``The Prisoner of Benda" \cite{ams},
an acclaimed episode of the animated television series
{\em Futurama}, features a two-body mind-switching machine. 
Any pair can enter the machine to swap minds, but there is
one serious limitation:
the machine will not work more than
once on the same pair of bodies.

After the Futurama community indulges in a mind-switching frenzy,
the question is raised: 
``Can the switching be undone so as to restore
all minds to their original bodies?"
The show provides an answer using
what is known in the popular culture as ``Keeler's theorem"
\cite{grime}.
The theorem is the brainchild of
the show's writer Ken Keeler \cite{aps}, 
who earned a PhD in Mathematics from Harvard
University in 1990 \cite{genealogy} 
before becoming a television writer/producer.
For ``The Prisoner of Benda," Keeler garnered a
2011 Writers Guild Award \cite{wga}.

The problem of undoing the switching can be modeled in terms of 
group theory. Represent the bodies involved in the switching frenzy by
$\{1,2,\dots, n\}$. The symmetric group $S_n$ consists of the
$n!$ permutations
of $\{1,2,\dots, n\}$. Let $I$ denote the 
identity permutation. A $2$-cycle $(ab)$ 
is called a transposition; it represents the
permutation which switches the minds of bodies $a$ and $b$.
The $k$-cycle $(a_1\dots a_k)$ is the permutation which sends 
$a_1$'s mind to $a_2$,
$a_2$'s mind to $a_3$, $\dots$, and $a_k$'s mind to $a_1$.  
Following the convention in \cite{BB}, we compute
products (i.e., compositions) in $S_n$ 
from right to left. For example, 
$(123) = (12)(23) = (13)(12) = (23)(13)$.  

The successive swapping of minds during the switching frenzy
can be represented by a product $P$ of
distinct transpositions in $S_n$.
(The transpositions must be distinct due to the limitation of the machine.)
In addition to viewing $P$ formally as a product, we can also view $P$ as
a permutation. It will be assumed that this permutation
is nontrivial, otherwise nothing
needs to be undone.  For an example of $P$, suppose that
$2$ switches minds with $3$ and then $2$ switches minds with
$1$; this corresponds to the product $P=(12)(23)$,  
yielding the mind-scrambling permutation $P=(123)$.  

To restore all minds to their original bodies,
one must find a product $\sigma$ of
distinct transpositions such that the permutation
$\sigma P$ equals $I$ and such that the transposition factors 
in the product $\sigma$ 
are distinct from those in the product $P$.  Such a 
$\sigma$ is said to {\em undo} $P$.  From now on, the phrase
``transposition factors" will be shortened simply to ``factors".

In the aftermath of a switching frenzy, 
the community may have
no recollection of the sequence of switches
that had taken place.
It is then expedient to find a product $\sigma$
that is guaranteed to undo the mind-scrambling permutation $P \in S_n$
{\em regardless} of which 
sequence of transpositions in $S_n$ had effected $P$.
Keeler's theorem explicitly produces such a product $\sigma \in S_{n+2}$.
Each factor in Keeler's $\sigma$ contains at least one entry in the set
$\{x,y\}$, where
\[
x:=n+1 \ \mbox{     and     } \ y:=n+2; 
\]
hence the factors in $\sigma$ are distinct
from whatever transpositions had effected $P$.
One can view $x$ and $y$ as altruistic {\em outsiders}
who had never entered the machine during the frenzy,
but who are subsequently willing to endure frequent mind switches
in order to help others restore their minds to their original bodies.

Viewed as a permutation, $P$ can be expressed (uniquely up
to ordering) as the product
$P = C_1 \cdots C_r$ of nontrivial disjoint cycles
$C_1,\dots, C_r$ in $S_n$ 
\cite[p. 77]{BB}.
For each $i=1,\dots,r$, let $k_i$ denote the length of cycle $C_i$.
In discussing Keeler's theorem and our refinement (Theorem 1),
we will assume that
$k_1+\cdots +k_r = n.$ 
This presents no loss of
generality, since if $k_1+\cdots +k_r = m < n$, then we could
relabel the bodies and mimic the arguments
using $m$ in place of $n$.

We now describe Keeler's method for constructing a product
$\sigma \in S_{n+2}$ which undoes $P = C_1 \cdots C_r$.
For convenience of notation, write $k=k_1$, 
so that $C_1$ is a $k$-cycle $(a_1 \dots a_k)$
with each $a_i \in \{1,2,\dots, n\}$.  It is easily checked that
$\sigma_1 C_1 = (xy)$, where $\sigma_1$ is the product of $k+2$
transpositions given by
\begin{equation}\label{eq:1}
\sigma_1 = (xa_1)(xa_2) \cdots (xa_{k-1})\cdot(ya_k)(xa_k)(ya_1).
\end{equation} 
For each $C_i$, define analogous products
$\sigma_i$ of $k_i + 2$ transpositions satisfying
\[
\sigma_i C_i = (xy) \ \mbox{   for   } \ i=1,\dots,r.
\]
Note that every factor of $\sigma_i$ has the form $(xu)$ or $(yu)$
for some entry $u$ in $C_i$.  Since disjoint cycles commute,
$(xy)$ commutes with every transposition in $S_n$, so
$\tau: = \sigma_r \cdots \sigma_2 \sigma_1$
is a product of distinct transpositions for which $\tau P = (xy)^r$.
Taking 
\begin{equation}\label{eq:2}
\sigma = 
\begin{cases} 
\ (xy) \tau, & \mbox{ if } \ r \mbox{ is odd }
\\
\ \tau, & \mbox{ if } \ r \mbox{ is even},
\end{cases}
\end{equation}
we find that $\sigma$ undoes $P$ and $\sigma$ is a product of distinct
transpositions in $S_{n+2}$,
each containing at least one entry in $\{x,y\}$, as desired.

By (\ref{eq:1}) and (\ref{eq:2}), 
the number of factors in Keeler's $\sigma$ is
either $n+2r+1$ or $n+2r$ 
according as $r$ is odd or even.
In Theorem 1 of the next section,  we refine Keeler's method by showing that
$P$ can be undone via a product of only
$n + r+2$ distinct transpositions,
each containing at least one entry in $\{x,y\}$.
We show moreover that this result
is ``best possible" in the sense that
$n + r + 2$ cannot
be replaced by a smaller number.
Thus Keeler's algorithm is optimal for $r=1$ and $r=2$, but for no other $r$.

With the aim of finding interesting classes of products
that can be undone using {\em fewer} than two outsiders,
we examined
what are undoubtedly the two most natural products $P$ in $S_n$ effecting 
the cycle $(12\dots n)$, namely  \cite[p. 81]{BB} 
\[
P_1 = (12)(23)(34)\cdots(n-1,n) \ \mbox{ and }
\ P_2 = (n-1,n)\cdots(3n)(2n)(1n).
\]
Theorems 2 and 3 determine how
many outsiders and how many mind switches are necessary and sufficient to  
undo each of these two products.
Theorem 2 shows that for $n \ge 5$, $P_1$ can be undone without any outsiders,
using only $n+1$ switches, where $n+1$ is best possible.
Theorem 3 shows
that for $n \ge 3,\ $ $P_2$ can be undone using 
only one outsider, again with
$n+1$ switches, where $n+1$ is best possible.

Suppose for the moment that $n \ge 5$. 
While $P_1$ and $P_2$ can both be undone
with fewer than two outsiders, 
there are other products $P_3(n)$ in $S_n$ effecting $(12\dots n)$
for which  {\em two} outsiders are required to undo $P_3(n)$. 
For an example with $n=5$, let
\[
P_3(5):=(54)(53)(52)(51)(12)(23)(14)(13)(24)(34) = (12345).
\]
Note that all ten transpositions in $S_5$ are factors of $P_3(5)$.
Suppose for the purpose of contradiction that $P_3(5)$ can be undone by
a product $\sigma$ in $S_6$, i.e., with just one outsider.  Every
entry in $P_3(5)$ must appear in $\sigma$, so
$\sigma$ must be a product of the five factors $(61)$, $(62)$, $(63)$,
$(64)$, $(65)$ in some order. The permutation $\sigma$ thus fails to fix
the entry $6$, which yields the contradiction  $\sigma P_3(5) \ne I$.
The argument for $n=5$ works
the same way for all $n\ge 5$ of the form $4k+1$ or $4k+2$.  Simply take
$P_3(n): = P_2J$, where $P_2$ is defined in Theorem 3, and $J$ is the identity
formulated as a product of all $\binom{n-1}{2}$ transpositions in
$S_{n-1}$, as in Theorem 4.  We omit the argument for
$n$ of the form $4k$ or $4k+3$, as it's a bit more involved.

The products $P_1$ and $P_2$ each have the property that no two
consecutive factors are disjoint.  In contrast, consider the product
of $m$ disjoint factors
\[
P(m):=(12)(34)\dots(2m-1,2m).
\]
We call $P(m)$ the {\em Stargate switch} because $P(2)$ represents
a sequence of mind swaps featured in an episode of the
sci-fi television series {\em Stargate SG-1}\cite{holiday}.  
The first and second authors \cite{EH} have given an optimal algorithm
for undoing $P(m)$; for $m>1$, the algorithm requires no outsiders.

When $n \ge 5$, Theorem 2 provides equalities of the form
$\sigma P_1=I$ which express the identity $I$
as a product of $2n$ distinct transpositions in $S_n$.
Such equalities lead to the question:
``What are necessary and sufficient conditions on $m$ and $n$ for
$I$ to be expressible as a product of $m$ distinct transpositions
in $S_n$?" Theorem 4 provides the answer: it is necessary
and sufficient that $m$ be an even integer with 
$6 \le m \le \binom{n}{2}$.  

In order to prove Theorems 2--4, we require some properties of
cycles proved via graph theory in Lemma 1. 
The proof of Lemma 1(c) incorporates an idea of Jacques Verstraete in a
proof due to Isaacs \cite{isaacs2}.
We are grateful for their permission to include
it here, as our original proof was considerably less elegant.

We will also need the well-known ``Parity theorem," which shows
that the identity
permutation $I$ cannot equal a product of an odd number of transpositions.
Two proofs of the Parity theorem may be found in \cite[pp. 82, 149]{BB};
for an elegant recent proof, see Oliver \cite{oliver}.

\section{An optimal refinement of Keeler's method}

Keeler's algorithm was designed to undo every mind-scrambling permutation 
$P=C_1\cdots C_r$
that is effected by an unknown sequence of mind swaps.  
In this section, we present another such algorithm.  
While Keeler's algorithm is optimal only for $r\le2$, we prove that
our algorithm is optimal for all $r$.

\begin{thm}\label{thm:1}

Let $P=C_1\cdots C_r$ be a product of $r$ disjoint $k_i$-cycles $C_i$
in $S_n$, with $k_i \ge 2$ and  $n=k_1+\cdots +k_r$.  
Define $x=n+1$ and $y=n+2$.
Then $P$ can be undone by
a product $\lambda$ of $n+r+2$ distinct transpositions
in $S_{n+2}$, each containing at least one entry in $\{x,y\}$.
Moreover, this result is best possible in the sense that $n+r+2$
cannot be replaced by a smaller number.
\end{thm}

\begin{proof}
Write $k=k_1$, so that
$C_1$ is a $k$-cycle $(a_1 \dots a_k)$.
Corresponding to the cycle $C_1$, define 
\[
G_1(x) = (a_1x)(a_2x)\cdots (a_kx) \ \mbox{  and  } \ F_1(x) =(a_1x).
\]
Corresponding to each cycle $C_i$ for $i=1,\dots,r$,
define $G_i(x)$ and $F_i(x)$ analogously.
Set
\[
\lambda =
(xy)\cdot G_r(x)\cdots G_2(x)\cdot (a_kx)G_1(y)(a_1x)\cdot 
F_2(y)\cdots F_r(y).
\]
It is readily checked that
$\lambda$ undoes $P$ and that
$\lambda$ is a product of $n+r+2$ distinct transpositions
in $S_{n+2}$, each containing at least one entry
in $\{x,y\}$.

It remains to prove optimality.
Suppose for the purpose of contradiction that
$P$ can be undone by
a product $\sigma$ of $t < n+r+2$ distinct transpositions
in $S_{n+2}$, each containing at least one entry
in $\{x,y\}$.
Then by the Parity theorem, $t \le n+r$.

On the other hand, we have the lower bound $t \ge n$,
since each of the $n$ entries in $P$ must occur 
(coupled with $x$ or $y$) in a factor of $\sigma$.
Let $A$
denote the set of entries in $C_1=(a_1 \dots a_k)$, and
let $a$ denote the leftmost element of $A$ appearing in the
product $\sigma$.  Since $P$ maps $a$ to some other element of $A$,
it follows that $a$ appears twice in $\sigma$, i.e.,  
$\sigma$ has both of the factors $(ax)$ and $(ay)$.
The same argument
shows that each of the $r$ cycles $C_i$ contains an entry which
appears twice in $\sigma$.  Thus the inequality $t \ge n$
can be strengthened to $t \ge n+r$.  Consequently,
$t = n+r$.  It follows that
each of the $r$ cycles $C_i$ contains exactly one entry which
appears twice in $\sigma$, and the other $n-r$ entries appear only once.
This accounts for all $n+r$ factors of  $\sigma$, so in particular,
$(xy)$ cannot be a factor of $\sigma$.

Let $a'$ denote the rightmost element of $A$ appearing in the
product $\sigma$.  Since $P$ maps some element of $A$ to $a'$,
it follows that $a'$ appears twice in $\sigma$.  Since $a$ is the
only element of $A$ that
appears twice  in $\sigma$, we must have $a=a'$.
Consequently, we have shown the following two properties of $C_1$:

\noindent (i) there is a unique entry $a$ in $C_1$ for which
the transpositions $(ax)$ and $(ay)$ both occur as factors of
$\sigma$, and 
(ii) each entry of $C_1$ other than $a$ occurs in exactly one
factor of $\sigma$, and that factor lies strictly
between $(ax)$ and $(ay)$.  

\noindent These two properties are similarly shared by
each of the $r$ cycles $C_i$.

Let $N_1$ denote the number of transpositions
in $\sigma$ that lie strictly between 
its factors
$(ax)$ and $(ay)$.  Define $N_i$
similarly for each of the $r$ cycles $C_i$. We may
assume without loss of generality that
$N_1 \le N_i$ for all $i$.
We may also assume that
the factor $(ax)$ in $\sigma$ lies to the left of
the factor $(ay)$, and that $a=a_k$.

Let $M_y$ denote the set of factors in $\sigma$
which contain the entry $y$ and which lie 
between $(a_kx)$ and $(a_ky)$ inclusive.
Suppose for the purpose of contradiction that
every transposition in $M_y$ has the form $(a_iy)$
for some $a_i \in A$.  Since
$\sigma$ must send $a_{i+1}$ to $a_i$ for each $i=1,\dots,k-1$,
it follows that the elements of $M_y$ 
have to occur in the following order in $\sigma$:
\[
(a_1y), \ (a_2y), \ \dots, \ (a_{k-1}y), \ (a_ky).
\]
But then $\sigma$ could not send $a_1$ to $a_k$, a contradiction.
Thus some transposition in $M_y$ must have the form $(hy)$, 
where $h \notin A$.
Consider the rightmost $(hy) \in M_y$ with $h \notin A$.
For some fixed $j > 1,\ $ $h$ is an entry of the cycle $C_j$.
Among all the elements $(a_iy) \in M_y$ 
that lie to the right of $(hy)$,
let $(a_my)$ denote the
one closest to $(hy)$.
As $\sigma$ cannot send $a_m$ to $h$, it follows that the entry $h$
occurs twice between  $(a_kx)$ and $(a_ky)$, i.e.,  $\sigma$ has factors
$(hx)$ and $(hy)$ both lying strictly between $(a_kx)$ and $(a_ky)$.
Thus $N_j < N_1$.  This violates the minimality of $N_1$,
giving us the desired contradiction.
\end{proof}

\section{A lemma on factorizations of cycles}

\begin{lem}\label{lem:1}
For $2 \le k \le n$,
suppose that the $k$-cycle $(a_1 \dots a_k) \in S_n$
equals a product $P$ of 
$t$ transpositions in $S_n$. Then
{\bf (a)}  $t \ge k-1,\ $ 
{\bf (b)} when $t=k-1$, the set of entries in $P$ is
$V:=\{a_1,\dots,a_k\}$, and 
{\bf (c)} when $t=k-1$, at least one factor of $P$ has the form
$(a_ia_{i+1})$  with $1 \le i < k$.
\end{lem}

\begin{proof}
Since $(ij)(ab)(ij)$ equals a transposition, a product of nondistinct
transpositions reduces to a shorter product of distinct transpositions.
Thus it suffices to prove the result when 
the factors of $P$ are distinct.
Let $W$ denote the set of entries in the product $P$.
Note that $W$ contains the set $V:=\{a_1,\dots,a_k\}$.
Define a graph $G$ with vertex set $W$ and with $t$ 
edges $[i,j]$ corresponding to the $t$ transposition
factors $(ij)$ of $P$.
Since $P=(a_1 \dots a_k)$ is a product of these 
$t$ transpositions, the graph $G$
has a connected component $H$ whose vertex set contains $V$.
A connected graph with $M$ vertices has at least $M-1$ edges
\cite[Theorem 11.2.1, p. 163]{cameron}, so $H$ and hence $G$ must have at least
$|V|-1=k-1$ edges.  Thus $t \ge k-1$.  This proves part (a).
(For another proof of part (a), see \cite[p. 77]{isaacs}. For a generalization
proved via linear algebra, see \cite{mackiw}.)

For the rest of this proof, suppose that $t = k-1$. 
Then $H$ has $t$ edges, so $G=H$ and
$G$ is connected.
If $V$ were {\em strictly} contained in $W$,
then again by \cite[Theorem 11.2.1, p. 163]{cameron},
$G$ would have at least $k$ edges.  Thus $V=W$,
which proves part (b).  (For a generalization of part (b), see
\cite{neuen}.)

To prove part (c), it remains
to prove that one of the $k-1$ edges of $G$ has the form
$[a_i,a_{i+1}]$ with $1 \le i < k$.
This is clear for $k=2$, so we let
$k \ge 3$ and induct on $k$.
A connected graph with $k$ vertices is a tree if and only if it has $k-1$
edges \cite[Theorem 11.2.1, p. 163]{cameron}.  Thus $G$ is a tree.
Let $(a_u a_v)$ denote the rightmost factor of $P$, with $u < v$.
Write $w=v-u$. If $w=1$, we are done, so assume that $w>1$.
Define the disjoint cycles
\[
r=(a_{u+1}\dots a_v) \ \mbox{  and  }  \ s=(a_1\dots a_u,a_{v+1}\dots a_k),
\]
so that $r$ is a $w$-cycle and $s$ is a $(k-w)$-cycle.
If $v=k$, then $s$ is interpreted as $(a_1\dots a_u)$, which in turn
is interpreted as the identity permutation when $u=1$.
Define $P'$ to be the product obtained from $P$ by removing the
rightmost factor $(a_ua_v)$. Let
$G'$ be the graph obtained from
$G$ by removing the edge $[a_u,a_v]$.  
Then $P'$ has $k-2$ factors and $G'$ has $k-2$ edges.
Since $P=sr(a_ua_v)$, we have $P' = sr$. It follows
that $G'$ is a forest of two trees $R$ and $S$, where
$R$ is a tree on the $w$ vertices $a_{u+1},\dots ,a_v$ and 
$S$ is a tree on the remaining vertices in $V$.
The $w$-cycle $r$ equals a product $Q$
of the $w-1$ factors of $P$ corresponding to the $w-1$ edges of $R$.
Since $w < k$, it follows by induction that $Q$, and hence $P$, 
has a factor of the required form $(a_ia_{i+1})$.
\end{proof}

\section{Optimal methods to undo 
\boldmath{$P_1$} and \boldmath{$P_2$}}
\begin{thm}\label{thm:2}
For $n \ge 5$, let $P_1$ denote the product of $n-1$
transpositions in $S_n$ given by
$P_1=(12)(23)(34)\cdots(n-1,n).$
There exists a product $\sigma$ of $n+1$ distinct transpositions in $S_n$
which undoes $P_1$,
and this result is best possible in the sense that
no such $\sigma$ can have fewer than $n+1$ distinct factors.
\end{thm}

\begin{proof}
Define
\[
\sigma = (3n)(2,n-1)(1n)(14)(2n)(13)\cdot (35)\cdots(3,n-1),
\]
where when $n=5$, the empty product 
$(35)\cdots(3,n-1)$
is interpreted as the identity.  
It is easily checked that $\sigma P_1=I$ and that
$\sigma$ is a product of $n+1$ distinct transpositions in $S_n$
all distinct from the $n-1$ transpositions in $P_1$.
It remains to prove optimality.

Suppose for the purpose of contradiction that there exists
a product $E$ of $k < n+1$ distinct transpositions in $S_n$
for which $E P_1=I$ and for which the $k$ transpositions in $E$
are distinct from the $n-1$ transpositions in $P_1$.
Since $E P_1=I$, the Parity theorem shows that $k \le n-1$.
On the other hand, 
since $P_1 = (12\dots n)$,
Lemma 1(a) gives $k \ge n-1$.  Thus
the number of transpositions in the product $E$ is exactly $n-1$.
Note that $E^{-1}$ is a product of these same $n-1$ transpositions
in reverse order, and $E^{-1} = P_1 = (12\dots n)$.  Hence by
Lemma 1(c), one of these $n-1$ transpositions in $E$ has the form
$(i,i+1)$ with $1 \le i < n$.
This contradicts the distinctness of the factors of $E$ from those in $P_1$,
since by definition, $P_1$ is a product of all $n-1$ transpositions
$(i,i+1)$ with $1 \le i < n$.
\end{proof}

\begin{thm}\label{thm:3}
For $n \ge 3$, let $P_2$ denote the product of $n-1$
transpositions in $S_n$ given by
$P_2=(n,n-1)\cdots(n3)(n2)(n1).$
There exists a product $\tau$ of $n+1$ distinct transpositions in $S_{n+1}$
which undoes $P_2$, 
and this result is best possible in the sense that
no such $\tau$ can have fewer than $n+1$ distinct factors.
\end{thm}

\begin{proof}
Define
\[
\tau = (2,n+1)(3,n+1)(4,n+1)\cdots (n,n+1)\cdot (1,2)(1,n+1).
\]
It is easily checked that $\tau P_2=I$ and that
$\tau$ is a product of $n+1$ distinct transpositions in $S_{n+1}$
all distinct from the $n-1$ transpositions in $P_2$.
It remains to prove optimality.

Suppose for the purpose of contradiction that there exists
a product $F$ of $k < n+1$ transpositions in $S_{n+1}$
for which $F P_2=I$ 
and for which the $k$ transpositions in $F$
are distinct from the $n-1$ transpositions in $P_2$.
Since $F P_2=I$, the Parity theorem shows that $k \le n-1$.
On the other hand, 
since $P_2 = (12\dots n)$,
Lemma 1(a) gives $k \ge n-1$.  Thus
the number of transpositions in the product $F$ is exactly $n-1$.
Note that $F^{-1}$ is a product of these same $n-1$ transpositions
in reverse order, and $F^{-1} = P_2 = (12\dots n)$.  Hence by
Lemma 1(b), the entries in these $n-1$ transpositions all lie
in the set $\{1,2,\dots,n\}$.  Since the permutation $F$ moves $n$,
it follows that
one of these $n-1$ transpositions in $F$ has the form $(in)$ with
$1 \le i < n$.
This contradicts the distinctness of the factors of $F$ from those in $P_2$,
since by definition, $P_2$ is a product of all $n-1$ transpositions
$(in)$ with $1 \le i < n$.  
\end{proof}

\begin{rem}
When $n=2$, two outsiders are required to undo
$P_1=P_2=(12)$, and an optimal $\sigma$ is given by
$(34)(23)(14)(24)(13)$.  In the cases $n=3$ and $n=4$, one outsider is
required to undo $P_1$, and optimal $\sigma$'s are given by
$(14)(13)(24)(34)$ and $(14)(25)(24)(35)(45)$, respectively.
\end{rem}

\section{$I$ as a product of \boldmath{$m$} 
distinct transpositions in \boldmath{$S_n$}}

\begin{thm}\label{thm:4}
For the identity $I$ to be expressible as a product of $m$
distinct transpositions in $S_n$, it is necessary and sufficient that
$m$ be an even integer with $6 \le m \le \binom{n}{2}$.
\end{thm}

\begin{proof}
We begin by showing that the conditions are necessary.  First, $m$ must
be even by the Parity theorem, and 
it is not hard to show that $m$ cannot equal 2 or 4.
Furthermore, $m$ cannot exceed $\binom{n}{2}$, 
since $\binom{n}{2}$ is the number of distinct transpositions in $S_n$.
This proves necessity, and it remains to show sufficiency.

Define 
$f(a,b,c)= (ac)(ab)(bc),$
which we view formally as a product
of 3 transpositions,  while noting that
$f(a,b,c)$ equals $(ab)$ when viewed as a permutation.
If a product $\lambda$ of transpositions has a factor $(ab)$,
then formally replacing $(ab)$ by $f(a,b,c)$ increases the number
of $\lambda$'s factors by 2, without altering
$\lambda$ as a permutation.

For even $m$ in the appropriate range, we now show
how to express $I$ explicitly as a product
of $m$ distinct transpositions in $S_4$, $S_5$, $S_6$, $S_7$, and $S_8$.
An analogous treatment will then inductively express $I$ as a product
of $m$ distinct transpositions in 
$S_{4k}$, $S_{4k+1}$, $S_{4k+2}$, $S_{4k+3}$, and $S_{4k+4}$,
for all $k \ge 2$, thus completing the proof.

For $m=6$, we have the base case
\[
I=(12)(23)(14)(13)(24)(34) \quad \mbox{in } S_4.
\]
This equality uses all six transpositions in $S_4$,
so to consider the values $m=8, 10$, we move up to $S_5$.
For $m=8$, replace the 
first transposition $(12)$ above by $f(1,2,5)$ to obtain
\[
I=(15)(12)(25)(23)(14)(13)(24)(34) \quad \mbox{in } S_5.
\]
For $m=10$, replace the transposition $(34)$ above by $f(3,4,5)$ to obtain
\[
I=(15)(12)(25)(23)(14)(13)(24)(35)(34)(45) \quad \mbox{in } S_5.
\]
This equality uses all ten transpositions in $S_5$,
so to consider the values $m=12, 14$, we move up to $S_6$.
For $m=12$, replace $(23)$ above by $f(2,3,6)$ to obtain
\[
I=(15)(12)(25)(26)(23)(36)(14)(13)(24)(35)(34)(45) \quad \mbox{in } S_6.
\]
For $m=14$, replace $(45)$ above by $f(4,5,6)$ to obtain
\[
I=(15)(12)(25)(26)(23)(36)(14)(13)(24)(35)(34)(46)(45)(56) \quad \mbox{in } S_6.
\]
This equality uses all of the fifteen transpositions in $S_6$
except for $(16)$, so to consider the values 
$m=16, 18, 20$, we move up to $S_7$.
For $m=16$, $m=18$, and $m=20$, successively replace
$(12)$ by $f(1,2,7)$, $(34)$ by $f(3,4,7)$,  and $(56)$ by $f(5,6,7)$,
respectively.
This yields the following for $m=20$:
\begin{eqnarray*}
I= &(15)(17)(12)(27)(25)(26)(23)(36)(14)(13)(24)(35) \times \\
& \quad \times \ (37)(34)(47)(46)(45)(57)(56)(67)  \quad \mbox{in } S_7.
\end{eqnarray*}
This equality uses all of the twenty-one transpositions in $S_7$
except for $(16)$, so to consider the values $m=22, 24, 26, 28$, 
we move up to $S_8$.
For $m=22$, $m=24$, and $m=26$, successively replace
$(23)$ by $f(2,3,8)$, $(45)$ by $f(4,5,8)$,  and $(67)$ by $f(6,7,8)$,
respectively.
This yields the following for $m=26$:
\begin{eqnarray*}
I= &(15)(17)(12)(27)(25)(26)(28)(23)(38)(36)(14)(13)(24)(35)(37) \times \\
& \quad \times  \ (34)(47)(46) (48)(45)(58)(57)(56)(68)(67)(78)  
\quad \mbox{in } S_8.
\end{eqnarray*}
This equality uses all twenty-eight transpositions in $S_8$ except
$(16)$ and $(18)$.  This suggests that we make the
atypical replacement of $(68)$ by $f(6,8,1)$ to obtain 
the following for $m=28$:
\begin{eqnarray*}
I= &(15)(17)(12)(27)(25)(26)(28)(23)(38)(36)(14)(13)(24)(35)(37)
\times \\ & \quad \times \ 
(34)(47)(46)(48)(45)(58)(57)(56)(16)(68)(18)(67)(78)  \quad \mbox{in } S_8.
\end{eqnarray*}
This equality uses all twenty-eight transpositions in $S_8$.
From here, we can repeat the procedure.
\end{proof}

\end{document}